\documentclass[11pt]{amsart}
\usepackage{amsmath, amssymb, amsthm, amscd, amsfonts}

\newtheorem{theorem}{Theorem}[section]
\newtheorem{lemma}[theorem]{Lemma}
\newtheorem{proposition}[theorem]{Proposition}
\newtheorem{corollary}[theorem]{Corollary}
\theoremstyle{definition}
 \newtheorem{definition}[theorem]{Definition}
\newtheorem{example}[theorem]{Example}

\theoremstyle{remark}

\numberwithin{equation}{section}

\begin{document}

\title[Approximation properties]{ Hilbert  $C^*$-bimodules of finite index and approximation properties of  $C^*$-algebras}

\author[M. Forough, M. Amini ]{Marzieh Forough, Massoud Amini}

\address{School of Mathematics, Institute for Research in Fundamental Sciences (IPM), P.O. Box 19395-5746, Tehran , Iran}
\email{mforough86@gmail.com, mforough@ipm.ir}

\address{Faculty of Mathematical Sciences, Tarbiat Modares University, Tehran 14115-134, Iran}
\address{School of Mathematics, Institute for Research in Fundamental Sciences, Tehran 19395-5746, Iran}
\email{mamini@modares.ac.ir, mamini@ipm.ir}



\begin{abstract}
Let $A$ and $B$ be arbitrary $C^*$-algebras, we prove that the existence of a Hilbert $A$-$B$-bimodule  of finite index  ensures that the WEP, QWEP,  and LLP along with other finite-dimensional approximation properties such as CBAP and (S)OAP are shared by $A$ and $B$. For this, we first study the stability of the WEP, QWEP and LLP under  Morita  equivalence of $C^*$-algebras.

 We present examples of Hilbert $A$-$B$-bimodules which are not of finite index, while such properties are shared between $A$ and $B$. To this end, we study twisted crossed products by amenable  discrete groups.
\end{abstract} \maketitle

\section{Introduction}
Morita equivalence between $C^*$-algebras was introduced by Rieffel in 70's and it has since become a standard tool in operator algebras. Two $C^*$-algebras $A$ and $B$ are called (strongly) Morita equivalent if there exists an imprimitivity $A$-$B$-bimodule which is a full left Hilbert $A$-module and a full right Hilbert $B$-module $X$ with inner products  $\langle .,.\rangle _{A}$ and $\langle .,.\rangle _{B}$ satisfying the compatibility condition $\langle x,y\rangle _{A}z=x\langle y,z\rangle _{B}$, for all $x, y, z \in X$.

Motivated by the fact that  the existence of an imprimitivity bimodule between two $C^*$-algebras gives an isomorphism between their $K_{*}$-groups, the notion of Hilbert $C^*$-bimodules  of finite (numerical) index was introduced in \cite{watatani2} and \cite{pinzari}.  In particular, the authors aimed to study a relation between the Jones index theory and $K$-theory for $C^*$-algebras.  Indeed, one can consider the notion of  Hilbert $A$-$B$-bimodules as a generalization of imprimitivity bimodules by replacing the above compatibility condition by a Pimsner-Popa type inequality. Imprimitivity bimodules are exactly the Hilbert  $C^*$-bimodules of finite index  such that both right and left indices are  equal to the identity. Watatani in \cite{watatani} introduced and studied the conditional expectations of finite index between unital $C^*$-algebras and Izumi in \cite{izumi} generalized this notion to the non-unital case. Hilbert $C^*$-bimodules associated to conditional expectations of finite-index  type provide typical examples of Hilbert $C^*$-bimodules of finite index.

Several properties of $C^*$-algebras are known to be stable under Morita equivalence. These include both finite-dimensional approximation properties and representation theoretic properties, c.f  \cite{skalski}, \cite{raeburn}. Regarding Hilbert $C^*$-bimodules of finite index as a generalization of imprimitivity bimodules, it is natural to ask which properties of $C^*$-algebras are preserved in the presence of Hilbert $C^*$-bimodules of finite index.  We show that finite-dimensional approximation properties such as  nuclearity, exactness, completely bounded approximation property (CBAP), as well as other important related properties like the local lifting property (LLP), weak expectation property (WEP), and quotient weak expectation property (QWEP) are shared between $C^*$-algebras $A$ and $B$ if there exists a Hilbert $A$-$B$-bimodule of finite index.  For this,  we first prove the stability of the LLP, WEP, and QWEP under Morita equivalence. Moreover, we show that the QWEP and nuclearity  are preserved in the presence of Hilbert $C^*$-bimodules of finite numerical index.

On the other hand, we give examples of $C^*$-algebras $A$ and $B$ such that there is no canonical Hilbert $A$-$B$-bimodule of finite index  while $A$ and $B$ share all the above mentioned properties.  Let $G$ be a discrete group acting on a unital $C^*$-algebra $A$ with a normal subgroup $H$. One can associate to the canonical conditional expectation $E:A \rtimes _{r}  G \rightarrow A \rtimes _{r} H$ a Hilbert $C^*$-bimodule $\mathcal{E}$. Then it can be concluded from \cite{izumi} and \cite{khoshkam} that the Hilbert $C^*$-bimodule $\mathcal{E}$ is of finite index if and only if $[G:H]$ is finite. When $G/H$ is  amenable  then we show that $A \rtimes_{r} G$ and $A \rtimes _{r} H$ share nuclearity, exactness, CBAP, (S)OAP, the WEP, QWEP and LLP, however $\mathcal E$ is not necessarily a Hilbert $C^*$-bimodule of finite index. For this,  we employ the notion of reduced twisted crossed products  and view $A \rtimes _{r}G$ as a reduced twisted reduced crossed product of $A \rtimes _{r}H$ by amenable group $G/H$. We investigate the preservation of  these properties under taking twisted crossed product by amenable discrete groups. To this end, we construct approximating nets in the case of twisted crossed products by discrete amenable groups.

We sketch the content of this paper.  In section 2, we discuss the stability of several properties of $C^*$-algebras, including the WEP, QWEP and LLP by Morita equivalence. In section 3, we first review some basic properties of Hilbert $C^*$-bimodules of finite index. Then we prove that if there exists a Hilbert $A$-$B$-bimodule of finite index then $A$ and  $B$ share finite-dimensional approximation properties like nuclearity, exactness, the CBAP, and (S)OAP along with the WEP, QWEP and LLP. Finally in section 4,  we show that if a discrete group $G$ acts on a unital $C^*$-algebra $A$ and $H$ is a normal co-amenable subgroup of $G$, then the $C^*$-algebras $A\rtimes _{r}H $ and $A\rtimes _{r}G$ share all the desired properties, while the associated canonical  Hilbert $C^*$-bimodule may fail to be of finite index. 

We use the standard notions of the theory of Hilbert $C^*$-modules. Let $X_{B}$ be a right Hilbert $C^*$-module over a $C^*$-algebra $B$.  We denote by $L(X_{B})$ the $C^*$-algebra of adjointable operators on $X_{B}$.  For $x, y \in X_{B}$, we put $\theta ^{r}_{x,y}(z)=x\langle y,z \rangle _{B}$ for $z \in X$ and call it a (right) rank one operator. We denote by $\mathbb{K}(X_{B})$ the $C^*$-algebra generated by rank one operators. Throughout this paper, we use the abbreviations c.p. and c.c.p. for completely positive and contractive completely positive maps. 
\section{Morita Equivalence}

In this section we prove that the WEP,  QWEP and LLP are invariant under Morita equivalence. To do this, we first recall some results from \cite{skalski} and \cite{smith}.

In Proposition 3.3 of \cite{skalski} it is  shown that certain finite-dimensional approximation properties including nuclearity, exactness, CBAP, (S)OAP are stable under Morita equivalence (for alternative proofs in the case of nuclearity and exactness, see \cite{raeburn} or \cite{katsura}). The proof relies on a factorization via algebras \cite[Lemma 3.1]{skalski}, instead of finite-rank factorizations,  followed by a construction due to Blecher \cite{blecher}. Indeed, we have the following criterion from Lemma 3.1 in \cite{skalski}.

 \begin{lemma}\label{skalski}
 Suppose that there exists an approximating net, that is, a net of c.c.p. maps $\varphi _{i}: A \rightarrow C_{i}$ and $\psi _{i}: C_{i} \rightarrow A$ with $\psi _{i}\circ\varphi _{i} \rightarrow id$, in point-norm topology. Consider the following finite-dimensional approximation properties: nuclearity, exactness, CBAP, (S)OAP. If  all $C_{i}$ have any of these properties then so does $A$, except in the case of the CBAP, where $A$ has the CBAP if $sup \Lambda (C_{i}) \leq \infty$.
 \end{lemma}

In \cite{blecher}, Blecher constructs completely contractive maps implementing the factorization of a Hilbert $C^*$-module $X$ over $A$ via the column modules $A^{\oplus n}$. These maps induce c.c.p. maps on the $C^*$-algebra $\mathbb{K}(X)$ of compact operators on $X$ which are factored  via $M_{n}(B)$. We recall the next result from \cite{smith}. 

\begin{proposition}\label{diag+morita}
Let $B$ be a $C^*$-algebra and $X$ be a right Hilbert $C^*$-module over $B$. Then there exist completely positive contractive maps $\varphi_{\alpha}: \mathbb{K}(X) \rightarrow M_{n(\alpha)}(B)$ and $\psi _{\alpha}: M_{n(\alpha)}(B)\rightarrow \mathbb{K}(X)$ such that $\|\varphi_{\alpha}\psi _{\alpha}(x)- x\|\rightarrow 0$, for all $x \in \mathbb{K}(X)$.
\end{proposition}

Next, we show the stability of the QWEP,  LLP and WEP under Morita equivalence of $C^*$-algebras. First, we consider the case of QWEP. For this, we require the following observation which follows from  Corollary 5.3 of \cite{ozawa}.

\begin{lemma}\label{ozawa}
Suppose that there exist c.c.p. maps $\varphi_{\alpha}: A \rightarrow C_{i}$ and $\psi _{\alpha}: C_{i} \rightarrow A$  such that $\|\varphi_{\alpha}\psi _{\alpha}(x)- x\|\rightarrow 0$, for all $x \in A$. If all $C^*$-algebras $C_{i}$ are QWEP, then so is $A$.
\end{lemma}

\begin{proposition} \label{ozawa+Morita}
Let $A$ and $B$ are Morita equivalent $C^*$-algebras. Then $A$ is QWEP if $B$ is $QWEP$.
\end{proposition}

\begin{proof}
Let $B$ be QWEP and $C^*$-algebra $A$ be Morita equivalent with $A$. Assume that $X$ is an imprimitive $A$-$B$-bimodule. It follows from Proposition \eqref{diag+morita} that there exist completely positive contractive maps $\varphi_{\alpha}: \mathbb{K}(X) \rightarrow M_{n(\alpha)}(B)$ and $\psi _{\alpha}: M_{n(\alpha)}(B)\rightarrow \mathbb{K}(X)$ such that $\|\varphi_{\alpha}\psi _{\alpha}(x)- x\|\rightarrow 0$, for all $x \in \mathbb{K}(X)$. Note that each $M_{n(\alpha)}(B)$ is QWEP since we assume that $B$ is QWEP. Now, employ Lemma \eqref{ozawa} to conclude that $\mathbb{K}(X)$ is QWEP. Observe that $A$ is $*$-isomorphic to $K(X_{B})$ so $A$ is QWEP, as desired.
\end{proof}

We recall that a pair of $C^*$-algebras $(A, C)$ is called nuclear if the algebraic tensor product $ A \odot C$ admits a unique $C^*$-tensor norm. Hence we can restate the Kirchberg's tensorial characterization of the LLP and WEP as follows: Suppose that $A$ is a unital $C^*$-algebra, then
 $A$ has the LLP if and only if $(A, B(l^{2}))$ is a nuclear pair and moreover,  $A$ has the WEP if and only if $(A, C^{*}(\mathbb{F}_{\infty}))$ is a unclear pair.

\begin{proposition}\label{Morita+WEP+LLP}
Suppose that  $A$ and $B$ are Morita equivalent $C^*$-algebras. Then $A$ and $B$ share the LLP or WEP.
\end{proposition}
\begin{proof}
Suppose that $C^*$-algebras $A$ and $B$ are Morita equivalent. Following the argument given in the proof of Theorem 15 of \cite{raeburn}, we conclude that  $(A, C)$ is a nuclear pair if and only if $(B,C)$ is a nuclear pair. The result can be concluded by this observation together with the Kirchberg tensorial factorization for the LLP and QWEP.

\end{proof}

We do not know if the lifting property  (LP) is stable under Morita equivalence, however we can obtain the following result.
\begin{proposition}
Let $A$ and $B$ be Morita equivalent separable $C^*$-algebras and let  $A$ have  the LP. Suppose $J$ is any closed ideal of a $W^*$-algebra $C$, then any c.c.p. map $\theta : B \rightarrow {C}/{J}$ has a c.c.p. lift.
\end{proposition}
\begin{proof}
Let $\theta : B \rightarrow \frac{C}{J}$ be a c.c.p. map from $B$ to a quotient of a $W^*$-algebra $C$. Since we assume that $A$ and $B$ are Morita equivalent, Proposition \eqref{diag+morita} implies that there exist c.c.p. maps $ \varphi _{\alpha}: B \rightarrow M_{n(\alpha)}(A)$ and $\psi _{\alpha} : M_{n(\alpha)}(A) \rightarrow B$ such that $\| \psi _{\alpha}\varphi _{\alpha}(x) -x\| \rightarrow 0$. Now consider the c.c.p. map $\varphi _{\alpha} \theta :  M_{n(\alpha)}(A)\rightarrow \frac{C}{J}$. Since  $M_{n(\alpha)}(A)$ has the LP (13.2.1 of  \cite{brown}), it has a c.c.p. lift $\sigma _{\alpha}$ i.e.,  $ \pi \sigma _{\alpha}= \theta \psi _{\alpha}$,  and so $ \| \pi \sigma _{\alpha} \varphi _{\alpha}(x)- \theta (x) \| \rightarrow 0$. Since $C$ is a $W^*$-algebra, we can apply Theorem 1.3.7 of \cite{brown} to obtain a cluster point  for the net of c.c.p. maps $\sigma _{\alpha} \varphi _{\alpha}: B \rightarrow C$ in the set of all bounded maps from $B$ into $C$ equipped with ultra-weak topology. Let $ \tilde{\theta}: B \rightarrow C$ be such a cluster point, then clearly $ \tilde{\theta}$ is a c.c.p. lift for $\theta$.
\end{proof}

Recall that an affirmative answer to the QWEP conjecture would imply that the LP and the LLP are equivalent for separable $C^*$-algebras (\cite{kirchberg}). Thus  by Proposition \eqref{Morita+WEP+LLP},  Morita equivalent separable  $C^*$-algebras share the LP if the QWEP conjecture holds.

\section{Hilbert $C^*$-bimodule of finite index}
We begin this section with a brief review of index theory for Hilbert $C^*$-bimodules from \cite{pinzari}. Let $A$ and $B$ be $C^*$-algebras and $_{A}X_{B}$  be an $A$-$B$-bimodule. We say that $X$ is a  right Hilbert $A$-$B$-bimodule if $X$ is a right Hilbert $B$-module and for all $a \in A$, the left multiplication map $\phi (a) : x \in X \mapsto ax \in X$ is adjointable, with adjoint $\phi(a)^{*}=\phi(a^{*})$. A left Hilbert $A$-$B$-bimodule can be defined in a similar manner. An  $A$-$B$-bimodule $X$ is called  bi-Hilbertian if it is endowed with a right as well as a left  Hilbert $A$-$B$-bimodule structure such that the Banach space norms coming from the corresponding inner products are equivalent.

Let $_{A} X_{B}$ be a bi-Hilbertian $C^*$-bimodule, and let $\lambda^{\prime} \geq 0$ satisfy  $\lambda ^{\prime} \|\langle x,x \rangle _{B}\|  \leq \| \langle x,x \rangle _{A} \|$, then for all $n \in N$ and $x_{1},...,x_{n} \in X$ we have  $ \lambda ^{\prime} \|\sum _{1}^{n} \theta ^{r}_{x_{i},x_{i}}\| \leq \|\sum _{1}^{n}  \langle x_{i},x_{i} \rangle_{A} \|$ (c.f. Proposition 2.4 of \cite{pinzari}). On the contrary, there may exist no $\lambda \geq 0$ for which  $\|\sum _{1}^{n}  \langle x_{i},x_{i} \rangle _{A}\| \leq \lambda \|\sum _{1}^{n} \theta ^{r}_{x_{i},x_{i}}\|$. Indeed the right numerical index is defined by the existence of such a constant  \cite{pinzari}.
\begin{definition}
 Hilbert $C^*$-bimodule $_{A}X_{B}$  is called of finite right numerical index if there exists $\lambda \geq 0$ such that for  all $n \in  \mathbb{N}$ and  all $x_{1},..,x_{n}\in X$, we have $\|\sum _{1}^{n}  \langle x_{i},x_{i} \rangle_{A}\| \leq \lambda \|\sum _{1}^{n} \theta ^{r}_{x_{i},x_{i}}\|$.
 \end{definition}

We  say that  $_{A}X_{B}$ is of finite left numerical index if the conjugate Hilbert $A$-$B$-bimodule $_{B}\bar{X}_{A}$ is of finite right numerical index.

If $_{A} X_{B}$ is a bi-Hilbertian $C^*$-bimodule of finite right numerical index, then there exists a unique norm continuous map $F: \mathbb{K}(X_{B}) \rightarrow A$ such that $F(\theta^{r}_{x,y})= \langle x,y \rangle_{A}$.  One can uniquely extend the maps $\phi : A \rightarrow L(X_{B})$ and $F$ to a unital $*$-homomorphism $\phi : A^{''}  \rightarrow \mathbb{K}(X_{B})^{''}$ and a normal positive map $F^{''}: \mathbb{K}(X_{B})^{''} \rightarrow A^{''}$ between the corresponding enveloping von Neumann algebras \cite{pinzari}.

\begin{definition}
If $_{A} X_{B}$ is of finite right numerical index, the right numerical index of $_{A} X_{B}$, denoted by $r-Ind[X]$,  is the element $F^{''}(I)$ in $A^{''}$. If $_{A}X_{B}$ is of finite left numerical index, the left index element of $X$ is $l-Ind[X]=r-Ind[\bar{X}]$.
\end{definition}

\begin{example}(c.f. Proposition 2.12 \cite{pinzari})\label{exx}
 Let $A \subseteq B$ be an inclusion of $C^*$-algebras with a conditional expectation $E: B \rightarrow A$ such that there exists $\lambda \geq 0$  satisfying $ \|E(b)\| \geq  \lambda \|b\|$, for all positive elements $b \in B$.  Consider $B$  as a  Hilbert $A$-$B$-bimodule with inner products ${}_A\langle x,y\rangle = E(xy^{*})$ and $\langle x, y\rangle _{B}=x^{*}y$. Here we shall denote this Hilbert $A$-$B$-bimodule by $\mathcal{E}$. Then by Theorem 1 of  \cite{frank}, there exists a $\lambda ^{\prime}$ such that $E-\lambda^{\prime}$ is completely positive. If one  chooses the best such $\lambda ^{\prime}$,  $\mathcal{E}$ is of  finite right numerical index with $\|r-Ind[\mathcal{E}]\|=\lambda ^{\prime -1}$.

Consider now $_{B}\mathcal{F}_{A}=B$ as a $B$-$A$-bimodule with inner products $\langle x, y\rangle _{B}=x^{*}y$ and $\langle x, y\rangle _{A}=E(x^{*}y)$. Then by the first part $\|r-Ind[\mathcal{F}]\|=1$. Note that $\mathcal{F}$ can be identified by $\bar{\mathcal{E}}$, and so $\|l-Ind[\mathcal{E}]\|= 1$.
\end{example}

\begin{theorem}
Let $_{A}X_{B}$ be a full Hilbert $A$-$B$-bimodule of finite right numerical index. Then $A$ is QWEP or nuclear if so is $B$.
\end{theorem}

\begin{proof}
First note that both nuclearity and the QWEP is preserved under Morita equivalence and the QWEP passes to double dual.  We Recall that a $C^*$-algebra is nuclear if and only if its double dual is injective. Since the right $B$-inner product is full, $B$ and $\mathbb{K}(X_{B})$ are Morita equivalent. Hence $\mathbb{K}(X_{B})^{''} $ is  QWEP  or injective if so is $B^{''}$.  Suppose that  $_{A} X_{B}$ is of finite numerical index then  Proposition 2.19 of \cite{pinzari} ensures the existence of a  conditional expectation $E: \mathbb{K}(X_{B})^{''} \rightarrow \phi^{''}(A^{''})$. Since injectivity and QWEP pass to sub $C^*$-algebras in the presence of a conditional expectation, $\phi^{''}(A^{''})$ is injective or QWEP if so is $\mathbb{K}(X_{B})^{''} $. Note that by Corollary 2.20 of \cite{pinzari}, $\phi ^{''}$ is faithful since we assume that  the linear span of the set of left $A$-inner products is dense. Hence $A^{''}$ is injective or QWEP if so is $\mathbb{K}(X_{B})^{''} $.  We recall that injectivity or QWEP of $A^{''}$ implies that $A$ is nuclear or QWEP, respectively. This completes the proof.

\end{proof}

\begin{corollary}
Let $_{A}X_{B}$ be a full Hilbert $A$-$B$-bimodule of finite numerical index. Then $A$  and $B$ share the QWEP and nuclearity.
\end{corollary}
\begin{corollary} Let $A \subseteq B$ be an inclusion of $C^*$-algebras and $E:B \rightarrow A$ be a conditional expectation of finite Pimsner-Popa index. Then $A$ is nuclear or $QWEP$ if and only if $B$ is so.
\end{corollary}
Let us recall the definition of Hilbert $C^*$-bimodules of finite index from \cite{pinzari}.
\begin{definition}
A bi-Hilbertian $A$-$B$-bimodule $_{A}X_{B}$  is called of finite right index if \\
\ \ $(i)$  $X$ is of finite right numerical index,\\
\ \ $(ii)$  $r-Ind[X] \in M(A)$.
\end{definition}

Similarly,  $X$ is of finite left index if the conjugate Hilbert $A$-$B$-bimodule $_{B}\bar{X}_{A}$ is of finite right  index. The module $_{A}X_{B}$ is called of finite index if it is of finite right as well as finite left indices.

\begin{example}\label{ex}
$(i)$ Let $A \subseteq B$ be an inclusion of $C^*$-algebras with a conditional expectation $E: B \rightarrow A$ such that there exists $\lambda \geq 0$  satisfying $ \|E(b)\| \geq  \lambda \|b\|$, for all positive elements $b \in B$.  By Example \ref{exx}, one can associate a Hilbert $A$-$B$-bimodule  $\mathcal{E}$ to the conditional expectation $E$. Then $\mathcal{E}$ is a Hilbert $A$-$B$-bimodule of finite index.\\
\ \ $(ii)$ (c.f. Proposition 1.8 of \cite{kaj2}) Let  $X$ be a Hilbert $A$-$B$-bimodule, and let $G$ be a countable discrete group. Suppose that $(A, B, X, G, \alpha, \beta, \gamma )$ be a $G$-equivariant system, i.e.,  $\alpha$, $\beta$,  and $\gamma$ are actions of $G$ on $A$, $B$ and $X$, respectively such that the following holds for all $g \in G$, $a \in A$, $b \in B$ and $x, y \in X$:
$$\alpha _{g} (_{A} \langle x, y \rangle) =_{A} \langle \gamma _{g}(x), \gamma_{g}(y) \rangle, \ \ \ \    \gamma_{g}(ax)=\alpha_{g}(a) \gamma _{g}(x),$$
$$\beta _{g} ( \langle x, y \rangle)_{B} = \langle \gamma _{g}(x), \gamma_{g}(y) \rangle _{B}, \ \ \ \    \gamma_{g}(xb)= \gamma _{g}(x) \beta_{g}(b).  $$

  Then $X \rtimes _{\gamma}G$ is a Hilbert $A\rtimes _{\alpha} G$-$B \rtimes _{\beta}G$-bimodule of finite index if so is $X$.  The same statement also holds for the reduced crossed products.
\end{example}

\begin{theorem} \label{index}
Let $A$ and $B$ be  $C^*$-algebras and $_{A} X_{B}$ be a Hilbert $C^*$-bimodule of finite  right index such that both left and right inner products are full. Then  the following properties pass from $B$ to $A$: nuclearity, exactness, CBAP, (S)OAP, WEP, LLP, and QWEP. 
\end{theorem}
\begin{proof}
 Since we assume the fullness of right inner product, we can regard $X$ as an imprimitivity $\mathbb{K}(X)-B$ bimodule that is $\mathbb{K}(X)$ and $B$ are Morita equivalent. Observe that nuclearity, exactness, CBAP, (S)OAP,WEP, LLP, and QWEP are preserved under  Morita equivalence. Thus $\mathbb{K}(X_{B})$ has any of the above mentioned properties if so does $B$.  Moreover, the assumption that  $_{A} X_{B}$  is of finite right index implies existence of a conditional expectation $E: \mathbb{K}(X_{B}) \rightarrow  \phi(A)$ by Corollary 2.29 \cite{pinzari}. Hence all of the properties pass  from $\mathbb{K}(X_{B})$ to $\phi(A)$.  On the other hand, by Corollary 2.28 of \cite{pinzari}, $\phi$ is faithful since the left inner product is full, and this proves that the following properties pass from $B$ to $A$: nuclearity, exactness, CBAP, (S)OAP, WEP, LLP, and QWEP.  
 
   \end{proof}

 \begin{corollary}
 Let $A$ and $B$ be  $C^*$-algebras and $_{A} X_{B}$ be a Hilbert $C^*$-bimodule of finite  index such that both left and right inner products are full. Then  $A$ and $B$ share the following properties: nuclearity, exactness, CBAP, (S)OAP, WEP, LLP, and QWEP.
 \end{corollary}
 
 We recall that in the case of  the inclusion of unital $C^*$-algebras $A \subseteq B$ with a conditional expectation of finite index,  it has been shown that the following properties are shared between $A$ and $B$: simplicity, self-absorption, the cancelation property, stable rank, nuclear dimension, see \cite{izumi, osaka1, osaka2, osaka3, popa1}.
\begin{corollary} \label{cond}
Let  $A \subseteq B$ be an inclusion of $C^*$-algebras and $E: B \rightarrow A$ be a conditional expectation of finite-index type. Then $A$ and $B$ share all the properties in Theorem \eqref{index}.
\end{corollary}

Let $\alpha$ be an action of a  finite group $G$ on a $C^*$-algebra $A$.  In \cite{phillips},  some conditions are presented under which the following  properties pass from $A$ to the fixed point algebra $A^{\alpha}$:  the ideal property, the projection property, topological dimension zero, and pure finiteness. Now, from   we can deduce a sufficient condition under which any of the  properties  passes from $A$ to the fixed point algebra $A^{\alpha}$.
\begin{corollary}
Let $\alpha$ be a saturated action (c.f.  Definition 2.8.5 of \cite{watatani}) of a finite group $G$ on a unital $C^*$-algebra $A$, then $A$ has any of the following properties: nuclearity, exactness, CBAP, (S)OAP, WEP, LLP and QWEP  if and only if the fixed point algebra $A^{\alpha}$ enjoys the same property.
\end{corollary}
\begin{proof}
By Proposition 2.8.6 of \cite{watatani}, the canonical conditional expectation from $A$ onto $A^{\alpha}$ is of finite-index type. Thus Corollary \eqref{cond} implies the result.
\end{proof} 
The following corollary is immediate from Example \eqref{ex} and Corollary \eqref{cond}.
\begin{corollary}\label{khoshkam}
Suppose that $G$ is a discrete group acting on a unital $C^*$-algebra $B$ and let $E: B \rightarrow A$ be a finite-index type conditional expectation from $B$ onto a unital $C^*$-subalgebra $A$ commuting with the action of $G$. Then $A \rtimes _{r} G$ and $B \rtimes _{r} G$ share all the above properties. The same holds for the full crossed products.
\end{corollary}

\section{Inclusions of crossed products}
In this section, we give an example of two $C^*$-algebras $A$ and $B$ sharing the properties nuclearity, exactness, CBAP, (S)OAP, the LLP,  WEP, and  QWEP, while the canonical  Hilbert $A$-$B$-bimodule  is not of finite index.

Let $\alpha$ be an action of a  discrete group $G$ on a unital $C^*$-algebra $A$, and $H$ be a normal subgroup of $G$.  We recall from \cite{khoshkam} that there is a natural conditional expectation $E$ from $A \rtimes _{ \alpha, r} G$ onto the $C^*$-subalgebra $A \rtimes _{\alpha, r} H$ such that $E(au_{s})=au_{s}$ if $s \in H$ and  $E(au_{s})=0$, otherwise.  Moreover, by Theorem 2.8 of \cite{khoshkam},  the conditional expectation $E$ is of finite-index type if and only if $[G:H] $ is finite.  So by part $(i)$ of  Example \eqref{ex} the Hilbert  ${A\rtimes _{\alpha, r} G} -{A\rtimes _{\alpha, r} H}$-bimodule $\mathcal{E}$ is of finite index if and only if $[G:H]$ is finite. In this case,  $A \rtimes _{\alpha, r}H $ and $A \rtimes _{\alpha, r}G$ share all of the properties mentioned in Theorem \eqref{index}.  We  show that this also happens when $G/H$ is amenable, while the Hilbert ${A\rtimes  _{\alpha, r}G}-{A\rtimes _{\alpha, r} H}$-bimodule $\mathcal{E}$ is not necessarily of finite index. To do this, we need the notion of the reduced twisted crossed product and stability of nuclearity, exactness, CBAP,  (S)OAP, LLP, WEP and QWEP by taking reduced twisted  crossed product by amenable discrete groups.

Let $A$ be a unital $C^*$-algebra and let $G$ be a discrete group. A twisted action of $G$ on $A$ is a pair $(\alpha, u)$ where $\alpha: G \rightarrow Aut(A)$ and $u: G \times G \rightarrow U(A)$ satisfy 

$(i)$ $\alpha_{s} \circ \alpha_{t}= Ad (u(s,t)) \alpha_{st}$,

$(ii)$ $\alpha_{r}(u(s,t))u(r,st)=u(r,s)u(rs,t)$,

$(iii)$ $u(1,s)=u(s,1)=1$,

\noindent for all $s,t,r \in G$. We shall refer to the quadruple $(A, G, \alpha, u)$ as a twisted $C^*$-dynamical system.

Following \cite{bedos},   one can construct the full and reduced twisted crossed products $A \rtimes _{(\alpha, u)} G$ and $A \rtimes _{(\alpha, u),r} G$.  To define the reduced twisted crossed product  for a twisted dynamical system $(A, G, \alpha, u)$  we consider the following covariant pair of representations: $\lambda_{u}: G  \rightarrow U(l^{2}(G, H))$ defined by
 $\lambda_{u}(s)(\xi)(l)=u(l^{-1},s)\xi(s^{-1}l)$ and $\pi_{\alpha}: A \rightarrow U(l^{2}(G, H))$ defined by $(\pi_{\alpha}(x))(\xi)(l)=\alpha_{l^{-1}}(x)(\xi)(l)$. It is easy to check that the covariance relation $\pi_{\alpha}(\alpha_{g}(x))=Ad (\lambda _{u}(g)) \pi_{\alpha}(x)$ holds for all $g \in G$ and $ x \in A$, and also $\lambda_{u}(s) \lambda_{u}(l)= \pi_{\alpha}(u(s,l))\lambda_{u}(sl)$ holds for all $l, s \in G$.\\

The next lemma should be well-known but we could not find a reference. For the sake of completeness, we give a proof here.
 \begin{lemma} \label{independence}
 The reduced twisted crossed product $A \rtimes_ {(\alpha, u),r} G$ does not depend on the choice of the faithful representation $A \subseteq B(H)$.
 \end{lemma}
 \begin{proof}
 Let $\{ \delta _{g}, g \in G\}$ be the canonical orthogonal basis for $l^{2}(G)$ and for any finite set $F \subseteq G$ let $P \in B(l^{2}(G))$ be the orthogonal projection onto the linear span of $\{ \delta _{g}, g \in F\}$.  Let $(e_{p,q})_{p,q \in F}$ be the canonical matrix unit in $P B(l^{2}(G))P \cong M_{F}(\mathbb{C})$. It is straightforward to check that
 $$(1 \otimes P) \pi_{\alpha}(a) (1 \otimes P)= (1 \otimes P) \pi_{\alpha}(a)=\sum _{ q \in F} \alpha_{q ^{-1}}(a) \otimes e_{q,q}$$
   and
 \begin{align*}(1 \otimes P) \lambda_{u}(s) (1 \otimes P)&= \sum_{p \in F}u(p^{-1},s) \otimes e_{p, s^{-1}p}.\end{align*}
 Therefore, if $x= \sum _{s} a_{s} \lambda _{u}(s) \in C_{c}(G, A)\subseteq B(H \otimes l^{2}(G)$, then, in $  A \otimes M_{F}(\mathbb{C})$, we have
 $$(1 \otimes P)  (\sum _{s} \pi_{\alpha}(a)\lambda_{u}(s)) (1 \otimes P)= \sum_{p \in F \cap sF} \alpha _{q^{-1}}(a) u(p^{-1},s) \otimes e_{p, s^{-1}p}. $$
 
 Now, the result follows by an argument as in the proof of Proposition 4.5.1 of \cite{brown}.
 \end{proof}
 Next we investigate the stability of the properties considered in Theorem \eqref{index}  under taking reduced twisted crossed product by amenable discrete groups. First we consider the finite-dimensional approximation properties  and the QWEP. For this, we use the criterion given in Lemma 3.1 of \cite{skalski}. To this end, inspired by Lemma 4.2.3 of \cite{brown}, we construct an approximating net for the twisted reduced crossed product by amenable discrete groups. First, we need the following lemmas.
 \begin{lemma}\label{basic2}
 Let $(\alpha, u)$ be a twisted action of a discrete group $G$ on a unital $C^*$-algebra $A$ with $u(s, s^{-1})=u(s^{-1}, s)=1$ and let $\iota$ be the trivial action of $G$ on a unital $C^*$-algebra $B$. Then,

 $(i)$ Suppose $(\alpha \otimes \iota, u \otimes 1)$ is the twisted action of $G$ on $A \otimes B$ satisfying $(\alpha \otimes \iota )(g) (a \otimes b)= \alpha (g)a \otimes b$. Then
 \begin{align*} (A \rtimes _{(\alpha, u), r}G) \otimes B &= (A \otimes B) \rtimes _{(\alpha \otimes \iota , u \rtimes 1), r} G. \end{align*}

 $(ii)$ Suppose $(\alpha \otimes _{max} \iota, u \otimes 1)$ is the twisted action of $G$ on $A \otimes _{max} B$ satisfying $(\alpha \otimes  _{max}\iota )(g) (a \otimes b)= \alpha (g)a \otimes b$. Then
 $$ (A \rtimes _{(\alpha, u)}G )\otimes B= (A \otimes _{max} B) \rtimes _{(\alpha \otimes _{max} \iota , u \rtimes 1)} G.$$
 \end{lemma}
 \begin{proof}
 The proof is similar to the non-twisted case, see Lemma 2.75 and Lemma 7.19 in \cite{williams}.
 \end{proof}
 By a  modification of Lemma 4.2.3 in \cite{brown}, we construct an approximating net in the twisted case.
 \begin{lemma}\label{Folner}
 Let $(A, G, \alpha, u)$ be a twisted $C^*$-dynamical system with $u(s^{-1}, s)=1$. Suppose that $F \subset G$ is a finite set, then there exist c.c.p. maps $\phi : A \rtimes _{(\alpha, u),r} G \rightarrow A \otimes M_{n}$ and $\psi : A \otimes M_{n} \rightarrow A \rtimes _{(\alpha, u),r} G$ such that $\psi \phi (\pi_{\alpha} (a) \lambda _{u}(s))= \frac{|F\cap sF|}{|F|}\pi_{\alpha} (a) \lambda _{u}(s)$.
 \end{lemma}
 \begin{proof}
 As in the proof of Lemma \eqref{independence}, define the c.c.p. map
 $$\phi : A \rtimes _{\alpha, u} G \rightarrow A \otimes M_{n},$$
  by $\phi (x)= (1 \otimes P) x (1 \otimes P)$ and so $\phi (a \lambda_{u}(s))= \sum_{p \in F \cap sF} \alpha _{p^{-1}}(a)u(p^{-1},s) \otimes e_{p, s^{-1}p} $.   Let
 $\psi : A \otimes M_{n} \rightarrow A \rtimes _{\alpha, u} G$ be defined by 
 $$\psi ( a \otimes e_{p,q})= \frac{1}{|F|}\pi_{\alpha}( \alpha _{p}(a) u(p, q^{-1})) \lambda _{u}(pq^{-1}).$$
Recall that any positive element in $A \otimes M_{n}$ is of the form $\sum_{ p, q \in F} a_{p}^{*} a_{q} \otimes e_{p,q}$. We show that $\psi(\sum _{p,q} a_{p}^{*} a_{q} \otimes e_{p,q})$ is a positive element in $A \rtimes _{(\alpha, u)} G$. To this end, observe that
\begin{align*}
( \sum _{p \in F} \pi_{\alpha}(a_{p}) \lambda _{u}(p ^{-1}))^{*}( \sum _{q \in F} \pi_{\alpha}(a_{q} )\lambda _{u}(q ^{-1}))&=
\sum _{p,q \in F} \lambda _{u}(p) \pi_{\alpha}(a_{p}^{*} a_{q})\lambda _{u}(q^{-1})\\
 &=\sum _{p,q \in F} \pi_{\alpha}(\alpha _{p}(a_{p}^{*} a_{q}) u(p,q^{-1}))\lambda _{u}(pq^{-1}),
 \end{align*}
Note that the last expression is equal to $\frac{1}{|F|} \psi(\sum_{ p, q \in F} a_{p}^{*} a_{q} \otimes e_{p,q})$. Thus we have shown that $ \psi$ is positive. By Lemma \eqref{basic2}, and a similar argument  as  in the proof of Lemma 4.2.3 of \cite{brown}, the  positivity of $\psi$ implies its complete positivity.

Finally, using the assumption that $u(p^{-1}, p)=1$ for all $p \in G$, we have
\begin{align*}
\psi \phi ( \sum _{s}\pi_{\alpha}(a_{s})\lambda_{u}(s))&=
\psi ( \sum _{s}\sum_{p \in F \cap sF} \alpha _{p^{-1}}(a) u(p^{-1},s) \otimes e_{p, s^{-1}p})\\
 &=\frac{1}{|F|}\sum _{s}\sum_{p \in F \cap sF}  \pi _{\alpha}( \alpha _{p}(\alpha _{p^{-1}}(a_{s}) )u(p^{-1},s)) \pi _{\alpha} (u(p,p^{-1}s))\lambda_{u}(s)\\
 &=\frac{1}{|F|}\sum _{s}\sum_{p \in F \cap sF}  \pi _{\alpha}(a_{s}u(p,p^{-1})u(1,s))\lambda_{u}(s)\\
 &=\frac{1}{|F|} \sum _{s}\sum_{p \in F \cap sF} \pi _{\alpha}(a_{s})\lambda_{u}(s)\\
 &= \frac{|F\cap sF|}{|F|}\sum_{s} \pi_{\alpha} (a_{s}) \lambda _{u}(s),
\end{align*}
 as desired. 
 \end{proof}
 Now, the required stability is concluded by the approximation nets constructed in the previous lemma.
 \begin{proposition}\label{twist+QWEP+finite}
  Let $(\alpha, u)$ be a twisted action of an amenable discrete group $G$ on  a unital $C^*$-algebra $A$ with $u(s^{-1},s)=1$.  Then  $A$ and $A \rtimes _{(\alpha, u),r} G$ share  any of the following properties: nuclearity, exactness, CBAP, (S)OAP and the QWEP. 
\end{proposition}
\begin{proof}
 Let $\{F_{\alpha} \}$ be a Folner sequence for amenable group $G$ and $\phi _{\alpha} : A \rtimes _{(\alpha, u),r} G \rightarrow A \otimes M_{n(\alpha)}$ and $\psi_{\alpha} : A \otimes M_{n(\alpha)} \rightarrow A \rtimes _{(\alpha, u),r} G$ be the corresponding c.c.p. maps constructed in Lemma \eqref{Folner} such that
 $$ \|\psi_{\alpha} \phi _{\alpha}(x) -x\|\rightarrow 0,$$
for any $ x \in A \rtimes _{\alpha, u} G$. Then $A$ and  $A \otimes M_{n(\alpha)}$ share all of the above  properties. Therefore, the result is immediate by Lemma 3.1 of \cite{skalski} and Lemma \eqref{Folner}.
\end{proof}
The Kirchberg's tensorial characterization of the LLP and WEP enables us to employ the argument given in Theorem 2.1 of \cite{farenick} to obtain the stability of the LLP and WEP under taking twisted reduced products by amenable discrete groups. For this,
 we first need the following lemma, which can be proved as Theorem 4.2.4 of \cite{brown}.
  \begin{lemma}\label{basic1}
  Let $A$ be a unital $C^*$-algebra and $\alpha$ be an action of a discrete amenable group $G$ on $A$ then $A \rtimes _{(\alpha, u)} G$ and $A \rtimes _{(\alpha, u),r} G$ are naturally $*$-isomorphic.
  \end{lemma}

 \begin{proposition} \label{twist+nuclear+pair}
Let $(\alpha, u)$ be a twisted action of  amenable discrete group $G$ on a unital $C^*$-algebra $A$ with $u(s^{-1},s)=1$ and let $C$ be a $C^*$-algebra.  Suppose that $(A, C)$ is a nuclear pair then $(A \rtimes _{(\alpha, u),r} G, C)$ is a nuclear pair.
\end{proposition}
\begin{proof}
 Let $(A, C)$ be a nuclear pair then 
\begin{align*}(A \rtimes _{(\alpha, u)} G) \otimes C &=(A \rtimes _{(\alpha, u), r} G) \otimes C\\
&=(A \otimes C) \rtimes _{( \alpha \otimes \iota), r} G\\
&=(A \otimes _{max} C)  \rtimes _{( \alpha \otimes _{max} \iota), r} G\\
&=(A \otimes _{max} C) \rtimes _{( \alpha \otimes _{max} \iota)} G\\
&= (A \rtimes _{(\alpha, u)} G) \otimes _{max}C.\end{align*}
Therefore, $(A \rtimes _{(\alpha, u),r} G, C)$ is a nuclear pair, as desired.
\end{proof}

 \begin{corollary} \label{twist+WEP}
Let $(\alpha, u)$ be a twisted action of  amenable discrete group $G$ on a unital $C^*$-algebra $A$ with $u(s^{-1},s)=1$. Suppose that $A$ has the LLP or WEP then $A \rtimes _{(\alpha, u),r} G$ has that property.
\end{corollary}
\begin{proof}
It follows from Proposition \eqref{twist+nuclear+pair} and Kirchberg's tensorial factorization for the LLP and WEP.
\end{proof}

Before stating the main result of this section, we need the following proposition which is analogous to Proposition 3 of \cite{bedos} for von Neumann algebra crossed products.
\begin{proposition}\label{seq}
Let $1 \rightarrow H \rightarrow G \rightarrow Q \rightarrow 1$ be an exact sequence of discrete groups and $\alpha$ be an action of a discrete group $G$ on a unital $C^*$-algebra $A$. Then there exists a twisted action $(\beta, u)$ of $Q$ on $A \rtimes _{\alpha |_{H},r}H$  with $u(s^{-1},s)=1$ for all $s \in G$, such that
$A \rtimes _{\alpha ,r} G$ is $*$-isomorphic to $(A \rtimes _{\alpha |_{H},r}H)\rtimes _{\beta, u}Q.$
\end{proposition}

Now, we are ready to state our main result of this section.
\begin{theorem}\label{main}
Let $1 \rightarrow H \rightarrow G \rightarrow Q \rightarrow 1$ be an exact sequence of  discrete groups with $Q$ amenable and $\alpha$ be an action of $G$ on a unital $C^*$-algebra $A$. Then $A \rtimes _{\alpha, r}G$ has any of the properties mentioned in Theorem \eqref{index}  if and only if $A \rtimes _{\alpha |_{H},r}H$ has that property.
\end{theorem}
\begin{proof}
The existence of a conditional expectation from $A \rtimes _{\alpha, r}G$ onto $A \rtimes _{\alpha |_{H},r}H$ ensures that if $A \rtimes _{\alpha, r}G$ has any of the above properties then $A \rtimes _{\alpha, r}G$ has that property. For the other direction, first note that by Proposition \eqref{seq}, there exists a twisted action $(\beta, u)$ of $Q$ on $A \rtimes _{\alpha |_{H},r}H$ such that $A \rtimes _{\alpha ,r} G$ is $*$-isomorphic to $(A \rtimes _{\alpha |_{H},r}H)\rtimes _{\beta, u}Q$.
Therefore, it is enough to use  Propositions \eqref{twist+QWEP+finite}, \eqref{twist+WEP} and \eqref{seq} to obtain the desired result.
\end{proof}

Let $H$ be a normal subgroup of a discrete group $G$ with $G/H$ amenable. Then Theorem \eqref{main} adds the LLP,  WEP and QWEP to the already known list of finite-dimensional approximation properties shared between $C^{*}_{r}(G)$ and $C^{*}_{r}(H)$.

\vspace{.3cm}
{\textbf{Acknowledgment.}}
The first author wishes to express her gratitude to Mehrdad Kalantar for many stimulating conversations and encouragement during her research visit to IMPAN. She would like to thank  the IMPAN for the hospitality during her visit. Authors were supported by a grant from IPM (Grant number. 94430215 for the second author).

\bibliographystyle{amsplain}

\end{document}